%% file: hittreesFinal.tex
\documentclass[a4paper]{article}
\usepackage{a4wide}
\usepackage{amsfonts,amsmath,amssymb,multicol}
\usepackage{graphicx}
\usepackage{url}
\usepackage{authblk}

\newcommand{\COLORON}{0}
\newcommand{\NOTESON}{0}
\newcommand{\Debug}{0}
\input{defs}

\newtheorem{theo}{Theorem}

\newtheorem{cor}[theo]{Corollary}

\theoremstyle{remark}
\newtheorem{rem}{Remark}
\newtheorem{prob}{Problem}

\newcommand{\defn}[1]{
\begin{equation*} {#1}  \end{equation*}
}

\title{Hitting Times, Cover Cost, and the Wiener Index of a Tree}

\author[1]{Agelos Georgakopoulos\thanks{Partly supported by FWF Grant P-24028-N18, EPSRC grant EP/L002787/1, and ERC grant 639046.}}
  
\author[2]{Stephan Wagner\thanks{Supported by the National Research Foundation of South Africa, grant number 70560.}}
\affil[1]{{Mathematics Institute}\\
 {University of Warwick}\\
  {CV4 7AL, UK}\\}
\affil[2]{Department of Mathematical Sciences\\ Stellenbosch University\\ Private Bag X1, Matieland 7602, South Africa\\ \url{swagner@sun.ac.za}}

\date{\today}

\begin{document}
\thispagestyle{plain}
\maketitle

\begin{abstract}
We exhibit a close connection between hitting times of the simple random walk on a graph, the Wiener index, and related graph invariants. In the case of trees we obtain a simple identity relating hitting times to the Wiener index.

It is well known that the vertices of any graph can be put in a linear preorder so that vertices appearing earlier in the preorder are ``easier to reach'' by a random walk, but ``more difficult to get out of''. We define various other natural preorders 
and study their relationships. These preorders coincide when the graph is a tree, but not necessarily otherwise.

Our treatise is self-contained, and puts some known results relating the behaviour or random walk on a graph to its eigenvalues in a new perspective.


\comment{ 
	We are particularly interested in the cover cost, i.e., the sum of all hitting times of the random walk starting at a given vertex $x$, and its counterpart, the reverse cover cost. We determine relations between these quantities for trees and determine their extremal values.
}
\end{abstract}

\section{Introduction and Statement of Results}

The \defi{hitting time $H_{vw}$} is the expected number of steps it takes a simple \rw\ on a graph $G$ to go from a vertex $v$ to a vertex $w$. The aim of this paper is to exhibit a close connection between hitting times, the Wiener index, and related graph invariants, especially for trees. The \defi{Wiener index} $W(G)$ of a graph $G$ is the sum  of the distances of all pairs of vertices of $G$:
\defn{W(G):=\sum_{\{x,y\} \subseteq V(G)} d(x,y) = \frac12 \sum_{x \in V(G)} \sum_{y \in V(G)} d(x,y).}
It has been extensively studied, especially for trees (see \cite{dobrynin2001wiener} and references therein), and has found applications in chemistry, communication theory and elsewhere. One of our main results is a rather surprising conection between the Wiener index of a tree and hitting times:

\begin{theo} \label{th1}
\Fe\ tree $T$ and every vertex $v \in V(T)$, we have 
$$\sum_{w\in V(T)} \left( H_{vw} + d(v,w) \right) = 2W(T).$$
\end{theo}

The sum of the left hand side is dominated by the first subsum $CC(x) := \sum_{w\in V(T)} H_{vw}$. In \cite{cc} this sum is dubbed the \defi{cover cost}, and it is argued that it is related to the \defi{cover time} of a graph, i.e., the expected time for a random walk starting at $v$ to visit all vertices. It is also the object of study in  \cite{PalaciosRenom}, in which lower bounds for $CC(x)$ are proved.
Defining the \defi{centrality} (also known as \defi{distance} of a vertex) $D(r):=\sum_{w\in V(T)}  d(r,w)$, the formula above can be rewritten in the following concise form:\\
%
\begin{equation}\label{CCW}
\fbox{$CC(v)+D(v) = 2W(T)$.}
\end{equation}
Note here also that
$W(T) = \frac12 \sum_{w \in V(T)} D(w)$.
The centrality $D(v)$ is a quantity of interest in combinatorial optimisation: a vertex where the centrality reaches its minimum appears with various names in the literature, including \defi{centroid}, \defi{barycenter} and \defi{median} (the latter in particular in weighted graphs). It is computable in linear time (by a straightforward breadth-first or depth-first search). The same is true for the Wiener index \cite{dankelmann1993computing}, and so we deduce that cover cost is computable in linear time.

In analogy to $CC(v)$ we also define the \defi{reverse cover cost} $RC(v) = \sum_{w \in V(G)} H_{wv}$. 
We will show that
\begin{theo} \label{RCW}
\Fe\ tree $T$ of order $n$ and every vertex $v \in V(T)$, the quantity  
$$RC(v) + (2n-1) CC(v) = 4(n-1)W(T)$$ 
is independent of $v$. Thus a vertex $v$ that maximizes $CC(v)$ minimizes $RC(v)$ and vice versa.
\end{theo}

Combining Theorems~\ref{th1} and~\ref{RCW} with elementary calculations we obtain the following formula that will be useful later:
\labequ{RCD}{RC(v) = (2n-1) D(v) - 2 W(T)}

\comment{ 
Both Theorem~\ref{th1} and Theorem~\ref{RCW} follow as immediate corollaries of the following formula for hitting times in a tree:

\begin{theo} \label{hit}
Let $x$ and $y$ be two vertices of a tree $T$ with $m$ edges. The hitting time $H_{xy}$ can be expressed as
$$H_{xy} = md(x,y) + D(y) - D(x).$$
\end{theo}
}

It is well known that the vertices of any graph can be put in a linear preorder $\leq$ \st\ vertices appearing earlier in the order are ``easier to reach  but difficult to get out of'', while vertices appearing later behave the other way around; more precisely, whenever $x\leq y$, we have $H_{yx}\leq H_{xy}$  \cite{CoTeWi,LovRan}. Note that it is not clear a priori that such a preorder exists, as there are about $n^2$ values $H_{xy}$ to be compared, but the preorder comprises only $n$ elements. Our next result shows that if the graph is a tree, then this ordering coincides with that of the values of $RC(x)$, the ordering of the values of $CC(x)$ reversed, as well as the orderings induced by further functions. In fact,  an alternative proof of
the existence of such a preorder is given by the equivalence of (ii) and
(iii) in Theorem~3, which will be shown to hold for arbitrary graphs in
Section~5.

We denote the degree of a vertex $w$ by $d(w)$ and define the  \defi{weighted centrality} $D_\pi$ by 
\defn{D_\pi(v):=\sum_{w\in V(T)} d(w) d(v,w)}
and the \defi{weighted cover cost} and \emph{weighted reverse cover cost} analogously by
\defn{CC_\pi(v):=\sum_{w\in V(T)} d(w) H_{vw},\qquad RC_{\pi}(v):= \sum_{w\in V(T)} d(w) H_{wv}.}
With this notation, we have
\begin{theo} \label{thBev}
\Fe\ tree $T$, and every pair of vertices $x,y\in V(T)$, the following are equivalent:
\begin{multicols}{2}
\begin{enumerate}
 \item \label{Bi} $D(x) \leq D(y)$;
 \item \label{Bii} $D_\pi(x) \leq D_\pi(y)$;
 \item \label{Biii} $H_{yx}\leq H_{xy}$;
 \item \label{Biv} $RC_\pi(x)\leq RC_\pi(y)$; 
 \item \label{Bv} $RC(x)\leq RC(y)$; 
 \item \label{Bvi} $CC(x)\geq CC(y)$.
\end{enumerate}
\end{multicols}
\end{theo}
The equivalence of \ref{Bi} to \ref{Bii} is an easy combinatorial observation, see \Sr{secTrees}. In fact, one easily finds $D_\pi(x) = 2D(x) - m$, where $m$ is the number of edges. 

The equivalence of \ref{Biii} and \ref{Biv} has been proved by Beveridge \cite{BevCen}\footnote{Beveridge \cite[Proposition~1.1]{BevCen} asserts a weaker statement, but the same proof applies. It is also proved there that the vertex minimising $D$ also minimises $RC_\pi$.}, but we will provide an alternative proof. The  equivalence of \ref{Bvi} to \ref{Bi} is an immediate consequence of \Tr{th1}, and \ref{Bvi} is equivalent to \ref{Bv} by \Tr{RCW}. The equivalence of \ref{Bii} to \ref{Biii} and \ref{Biv} actually holds in greater generality (see \Sr{secEx}): if one replaces $D_\pi$ by an appropriate quantity, then it remains true for arbitrary graphs. However, the equivalences of \ref{Bii}, \ref{Biii} and \ref{Biv} are the only ones that remain true for arbitrary graphs, as we prove in \Sr{secEx}.


\medskip

The results above can be interpreted as follows: there is a simple  
function $D: V(T)\to \R_+$ the values of which determine the cover cost and reverse cover cost. 
It is natural to ask whether something similar holds for general graphs. Our next result shows that this is indeed the case. Taking advantage of the theory of the relationship between \rw{s} and electrical networks \cite{DoyleSnell,TetRan,XuYau}, we use the following parameters   that can be thought of as generalisations of $D$ and $D_\pi$: let $r(v,w)$ denote the effective resistance between two vertices $v$ and $w$, and define the \defi{resistance-centrality} and \defi{weighted resistance-centrality} by
\defn{R(v):= \sum_{w\in V(T)} r(v,w),\qquad R_\pi(v):= \sum_{w\in V(T)} d(w) r(v,w).}
A well-known generalisation of the Wiener index for non-trees is the \defi{Kirchhoff index} \cite{klein1993resistance,mohar1993novel} or \defi{quasi-Wiener index}, defined as
\defn{K(G) := \sum_{\{x,y\} \subseteq V(G)} r(x,y) = \frac12 \sum_{x \in V(G)} \sum_{y \in V(G)} r(x,y).}
We also define the weighted variants 
\defn{K_\pi(G) :=  \frac12 \sum_{x \in V(G)} \sum_{y \in V(G)} d(y) r(x,y)\qquad \text{and}\qquad K_{\pi^2}(G) :=  \frac12 \sum_{x \in V(G)} \sum_{y \in V(G)} d(x) d(y) r(x,y).}

\begin{theo} \label{thMain}
\Fe\ connected graph $G$, and every vertex $x\in V(G)$, we have 
\begin{align*}
CC(x)&=mR(x) - \frac{n}{2} R_\pi(x)+ K_\pi(G),\\ 
RC(x)&=m R(x) + \frac{n}{2} R_\pi(x) - K_\pi(G),\\
RC_\pi(x)&= 2mR_\pi(x) - K_{\pi^2}(G), \text{ and}\\
CC_\pi(x)&= K_{\pi^2}(G).
\end{align*}
\end{theo}
Note that unlike trees, the three orderings according to $CC,RC$ and $RC_\pi$ are determined by three different functions, namely the functions $D_1(x)=2m R(x) -nR_\pi(x), D_2(x)= 2m R(x) +nR_\pi(x)$  and $D_3(x)= R_\pi(x)$ respectively (all of which are themselves determined by the two functions $R$ and $R_\pi$). This does not a priori mean that these orderings are different, since there is strong dependence between these functions. We will however construct examples showing that no two of these orderings always coincide. 

The fact that $CC_\pi(x)$ is constant is well-known, especially when $CC_\pi(x)$ is expressed as the expected hitting time from $x$ to a random vertex $y$  chosen according to the stationary distribution of \rw\ \cite{AldousFill,KemSne}; moreover, this constant, which is known as the \defi{Kemeny constant}, can be expressed in terms of the eigenvalues of the matrix $M$ of transition probabilities of \g ($m_{ij}= 1/d_i$ if $ij\in E(G)$ and 0 otherwise) as $CC_\pi(x) = {2m} \sum_{\lambda \neq 1} \frac1{1-\lambda}$, where $\lambda$ runs over all eigenvalues $\neq 1$ of $M$, see \cite[Formula 3.3]{LovRan}. It was observed in \cite{CheZhaRes} that the latter expression $ {2m} \sum_{\lambda \neq 1} \frac1{1-\lambda}$ equals  $K_{\pi^2}(G)$, but apparently the resulting fact that $CC_\pi(x) = K_{\pi^2}(G)$ has not been noticed before. It is thus worth pointing out this triple equality:

$$CC_\pi(x) = {2m} \sum_{\lambda \neq 1} \frac1{1-\lambda}= K_{\pi^2}(G).$$

A similar formula is known for the Kirchhoff index:
$$K(G) = n \sum_{\mu \neq 0} \frac{1}{\mu},$$
the sum being over all nonzero Laplacian eigenvalues $\mu$ of $G$, see \cite{dobrynin2001wiener,merris1989edge,mohar1991eigenvalues}. We show that both these eigenvalue formulas can be proven along the same lines, and derive analogous formulas for $R_{\pi}(x)$ and $R(x)$, see Theorem~\ref{thm:eigen} below.


The cover cost $CC(r)$ was proposed in \cite{cc} as a tractable variant of the \defi{cover time} $CT(r)$ ---i.e., the expected time for a \rw\ from $r$ to visit all other vertices of the graph--- which is much harder to compute. 
Combining \Tr{th1} and \Tr{RCW} with results of Aldous \cite{Aldous} and Janson~\cite{Janson}, we deduce  that for uniformly random rooted labelled trees $(T,r)$, the expected value of $CC(r)/|V(T)|$ is of the same asymptotic order as the expected value of the cover time $CT(r)$. This is related to a conjecture of Aldous, see \Sr{ranT}.

Using \Tr{th1} we are able to find the extremal rooted trees for the cover cost: in \Sr{sec:ext} we prove that, for a fixed number of vertices, $CC(x)$ is minimised by the star rooted at a leaf, and maximised by the  path rooted at a midpoint. It turns out that the same rooted trees are extremal for the cover time as well, by theorems of Brightwell \& Winkler \cite{BW2} and Feige \cite{FeiCol} respectively. Moreover, the same trees are extremal also for the Wiener index \cite{dobrynin2001wiener,entringer1976distance}. The hitting time on its own turns out to be, not surprisingly, maximised by the two endpoints of a path (if only trees are considered).

As a further application of our results, we obtain a precise description of the behaviour of $CC$ and $RC$ for random rooted trees (labelled trees, or more generally trees from a simply generated class). Interestingly, the average cover cost is of order $n$ times the average cover time of such a tree, which had been shown by Aldous \cite{Aldous} to be of order $n^{3/2}$; see \Sr{ranT} for more.

\section{Preliminaries}\label{secPrel}

All graphs in this paper are finite and simple.
A {\em random walk} on a graph
$G$ begins at some vertex and when at vertex $x$, traverses one of the edges incident to $x$ according to the uniform probability distribution.

Any finite graph can be seen as a (passive, resistive) electrical network, by considering each edge as a unit resistor, and there is a well-known theory relating the behaviour of the \rw\ on a graph to the behaviour of electrical currents \cite{DoyleSnell,LyonsBook}. We exploit this relationship in this paper by using the following formula of Tetali \cite{TetRan}, expressing hitting times in terms of effective resistances.
\labtequc{tet}{$H_{xy} = \frac12 \sum_{w \in V(G)} d(w) (r(x,y) + r(w,y) - r(w,x)) = mr(x,y) + \frac12 \big(R_{\pi}(y) - R_{\pi}(x)\big)$.}
Here, $r(x,y)$ denotes the \defi{effective resistance} between $x$ and $y$, and can be defined as the potential difference between $x$ and $y$ induced by the unique \flo{x}{y}\ of intensity 1 satisfying \kcl; see \cite{AgCurrents} for details.

Tetali's formula \eqref{tet} is easiest to use when considering sums or differences that cause some of its terms to cancel out. Fore example, the well-known formula of Chandra et al.\ \cite{CRRST} expressing the \defi{commute time} $\kappa_{xy}:= H_{xy} + H_{yx}$ in terms of the effective resistance follows immediately:
\labtequc{com}{$H_{xy} + H_{yx} = 2m r(x,y)$.}

\section{Results for all graphs}
Our first goal will be the proof of Theorem~\ref{thMain}, our results on trees will follow by specialisation. The main tool we will use is Tetali's formula \eqref{tet}; using this, we can express $CC(x)= \sum_y H_{xy}$ as
\begin{align*}
CC(x) &= \frac12 \sum_{y \in V(G)} \sum_{w \in V(G)} d(w) (r(x,y) + r(w,y) - r(w,x)) \\
&=\frac12 \sum_{w \in V(G)} d(w) \sum_{y \in V(G)} r(x,y) + \frac12 \sum_{y \in V(G)} \sum_{w \in V(G)} d(w)r(w,y) -\frac12 \sum_{w \in V(G)} d(w) r(w,x) \sum_{y \in V(G)} 1 \\
&= \frac12 \cdot 2m \cdot R(x) + K_{\pi}(G) - \frac12 \cdot R_{\pi}(x) \cdot n \\
&= mR(x) - \frac{n}2 \cdot R_{\pi}(x) + K_{\pi}(G).
\end{align*}
The proofs of the other three identities in Theorem~\ref{thMain} are similar.

\medskip
While the weighted cover cost $CC_{\pi}(x)$ is independent of the vertex $x$, this is not the case for the ordinary cover cost $CC(x)$. If, however, the graph is regular, then the cover cost is clearly also constant (since we have $CC_{\pi}(x) = kCC(x)$ on a $k$-regular graph), which has already been pointed out by Palacios \cite{PalKir}. The following theorem shows that the converse is also true.

\begin{cor} \label{correg}
The cover cost $CC(x)$ is independent of the starting vertex $x$ \iff\ \g is regular. In this case, we have 
$$CC_{\pi}(x) = K_{\pi^2}(x) = k CC(x) = k K_\pi(G)=k^2 K(G),$$
where $k$ is the vertex degree.
\end{cor} 

\begin{proof}
We claim that, for every connected graph \G, and every vertex $x$ of \G, we have 
\labtequc{zw}{$\sum_{z \sim x} (CC(x) - CC(z)) = nd(x) - 2m$.}
Note that this claim implies that if $CC(x)$ is independent of the starting vertex $x$ then \g is regular, for the left hand side is 0 in that case. To prove \eqref{zw}, we write
\begin{align*} 
\sum_{z \sim x} (CC(x) - CC(z)) &= \sum_{z \sim x} \sum_y (H_{xy} - H_{zy})\\
&= \sum_{z \sim x} \sum_{y\neq x} (H_{xy} - H_{zy} ) + \sum_{z \sim x} (0 - H_{zx} ) 
\end{align*}
Now note that for $y\neq x$ we have $H_{xy}= 1 + \frac1{d(x)} \sum_{z \sim x}H_{zy}$, since the \rw\ from $x$ moves to one of its neighbours $z$ in its first step. Rearranging this we obtain $\sum_{z \sim x}  (H_{xy} - H_{zy} )=d(x)$. The return time $H^+_{xx}$ to $x$, i.e., the expected time for a \rw\  from $x$ to reach $x$ again, is given by $H^+_{xx}= 2m/d(x)$ \cite[Lemma~1]{BW2}. Using this, and an argument similar to the one above, we obtain $\sum_{z \sim x} - H_{zx} =d(x)-2m$. Plugging these two equalities into the sum above yields \eqref{zw}.

\medskip
Suppose, conversely, that \g is $k$-regular. Then it follows immediately from Theorem~\ref{thMain} that
$$CC_{\pi}(x) = K_{\pi^2}(x) = k CC(x) = k K_\pi(G)=k^2 K(G),$$
as desired.
\end{proof}

It seems to be much harder to characterise those graphs for which the reverse cover cost or the weighted reverse cover cost are constant. Clearly this is the case for transitive graphs, but there might be other examples as well:

\begin{prob}
For which graphs are $RC(x)$ or $RC_{\pi}(x)$ independent of the vertex $x$?
\end{prob}

It is noteworthy that the quantities involved in Theorem~\ref{thMain} can be represented in terms of eigenvalues of matrices associated with the graph $G$. The following theorem makes this more explicit -- two of the identities have already been mentioned in the introduction, we give their short proofs to show the analogy.

\begin{theo}\label{thm:eigen}
For a matrix $A$, let $\mathcal{E}(A)$ denote the set of eigenvalues of $A$. Let $L$ be the Laplacian matrix of a graph $G$, let $M$ be the matrix of transition probabilities and $N = I - M$. For a given vertex $v$, let $L_v$ and $N_v$ be the matrices obtained from $L$ and $N$ by removing the row and column that correspond to $v$. The quantities $K(G)$, $K_{\pi^2}(G)$, $R(G)$ and $R_{\pi}(G)$ can be expressed in terms of eigenvalues of these matrices as follows:
\begin{enumerate}
\item $K(G) = n \sum_{\lambda \in \mathcal{E}(L) \setminus \{0\}} \frac{1}{\lambda}$,
\item $K_{\pi^2}(G) = 2m \sum_{\lambda \in \mathcal{E}(N) \setminus \{0\}} \frac{1}{\lambda}$,
\item $R(v) = \sum_{\lambda \in \mathcal{E}(L_v)} \frac{1}{\lambda}$,
\item $R_{\pi}(v) = \sum_{\lambda \in \mathcal{E}(N_v)} \frac{1}{\lambda}$.
\end{enumerate}
\end{theo}

\begin{proof}
Let $L_{vw}$ be the matrix that is obtained from $L$ by removing the row and column associated to $v$ and $w$. The key tool of our proof is the fact that $\det L_v$ equals the number $\tau(G)$ of spanning trees of $G$, while $\det L_{vw}$ is the number of so-called thickets: spanning forests consisting of two components, one of which contains $v$, the other $w$ \cite[Proposition~14.1]{Biggs}. 
The effective resistance is the quotient of the two, see \cite[Chapter~17]{Biggs}:
$$r(v,w) = \frac{\det L_{vw}}{\det L_v} = \frac{\det L_{vw}}{\tau(G)}.$$
Summing over all $v,w$, we obtain
$$K(G) = \frac{1}{\tau(G)} \sum_{\{v,w\} \subseteq V(G)} \det L_{vw}.$$
The sum is (up to sign) the coefficient of $t^2$ in the characteristic polynomial $\det(t I - L)$ of $L$, while the coefficient of $t$ is well known to be (up to sign) $n \tau(G)$: the sum of all the determinants $\det L_v$, which are all equal to $\tau(G)$. In both instances, the sign merely depends on the parity of the number of vertices. Our first formula now follows immediately from Vieta's theorem.

For the second equation, note that $N$ results from $L$ by dividing each row by the degree of its corresponding vertex. It follows immediately from the properties of the determinant that
$$\det N_v = \det L_v \prod_{x \in V(G) \setminus \{v\}} d(x)^{-1} = P d(v) \tau(G),$$
where $P = \prod_{x \in V(G)} d(x)^{-1}$, and likewise
$$\det N_{vw} = P d(v)d(w) \det L_{vw} = Pd(v)d(w) \tau(G) r(v,w).$$
By the same argument as before, we obtain
$$\sum_{\lambda \in \mathcal{E}(N) \setminus \{0\}} \frac{1}{\lambda} = \frac{\sum_{\{v,w\} \subseteq V(G)} \det N_{vw}}{\sum_{v \in V(G)} N_v} = \frac{P\tau(G)\sum_{\{v,w\} \subseteq V(G)} d(v)d(w)r(v,w)}{P\tau(G)\sum_{v \in V(G)} d(v)} = \frac{K_{\pi^2}(G)}{2m},$$
and our second formula follows.

Next we notice that $\sum_{\lambda \in \mathcal{E}(L_v)} \frac{1}{\lambda}$ is the quotient of the linear and the constant coefficient of the characteristic polynomial of $L_v$, which in turn equals
$$\frac{\sum_{w \in V(G) \setminus \{v\}} \det L_{vw}}{\det L_v} = \frac{\tau(G) \sum_{w \in V(G)} r(v,w)}{\tau(G)} = R(v),$$
proving our third statement. The fourth follows analogously.
\end{proof}

\subsection{A characteristic polynomial}
While it seems that $K_{\pi}(G)$ cannot be expressed in terms of eigenvalues of a matrix, there is an alternative way to express $K(G)$ as well as its weighted analogues in terms of coefficients of a polynomial: define
$$P(u,v) = \det(uI + vD - L),$$
where $I$ is the identity matrix, $D$ the diagonal matrix whose entries are the degrees of $G$, and $L$ the Laplacian matrix of $G$. Note that $P(0,0) = 0$, $P(u,0)$ is the characteristic polynomial of $L$, and $P(0,v)$ is a constant multiple of the characteristic polynomial of $N$. Moreover, we have the following relations (which are obtained in the same way as Theorem~\ref{thm:eigen}):
$$[u]P(u,v) = (-1)^{n-1} n\tau(G),\qquad [v]P(u,v) = 2(-1)^{n-1}m\tau(G),\qquad [u^2]P(u,v) = (-1)^n \tau(G)K(G),$$
$$[uv]P(u,v) = 2(-1)^n \tau(G)K_{\pi}(G),\qquad [v^2]P(u,v) = (-1)^n \tau(G)K_{\pi^2}(G).$$

\section{Trees} \label{secTrees}
In the case of trees, the effective resistance between two vertices equals their distance. This and some other special properties of trees cause the formulas in Theorem~\ref{thMain} to simplify greatly. Specifically, for any tree $T$, we have $K(T) = W(T)$ (i.e., the Kirchhoff index equals the Wiener index) as well as
$$K_{\pi}(T) = \frac12 \sum_{x \in V(T)} \sum_{y \in V(T)} d(y)d(x,y) = \frac14 \sum_{x \in V(T)} \sum_{y \in V(T)} (d(x)+d(y))d(x,y) = \frac12 \operatorname{Sch}(T)$$
and
$$K_{\pi^2}(T) = \frac12 \sum_{x \in V(T)} \sum_{y \in V(T)} d(x)d(y)d(x,y) = \operatorname{Gut}(T).$$
The quantities $\operatorname{Sch}(T)$ and $\operatorname{Gut}(T)$ in the two equations above are known as the \defi{Schultz index} and the \defi{Gutman index} respectively. It is known that for a tree $T$ of order $n$, one has (cf. \cite{Gutman})
$$\operatorname{Sch}(T)  = 4W(T) - n(n-1)\qquad \text{and}\qquad \operatorname{Gut}(T) = 4W(T) - (n-1)(2n-1).$$
Both identities can be proven along the same lines as the following lemma:
\begin{lemma}\label{treelemma}
For any vertex $x$ of a tree $T$, we have
\labtequc{Dpi}{$\sum_{y \in V(T)} d(y)d(x,y) = D_\pi(x) = 2 D(x) - m = 2\sum_{y \in V(T)} d(x,y) - m$.}
\end{lemma}
\begin{proof}
To show \eqref{Dpi}, we will check that any edge $e$ has the same \defi{contribution} to the two sides of the equation, where we think of the contribution of $e$ as the number of times we add a term $d(x,y)$ \st\ $e$ lies on the \pth{x}{y} (and thus contributes one unit to the distance). To this end, let $A_x(e)$ be the set of vertices on the same side of $e$ as $x$ and $B_x(e) = V(T) \setminus A_x(e)$ the complement. Then the contribution of $e$ to $D_\pi(x) := \sum_w d(w) d(x,w)$ is, by definition, $\sum_{w\in B_x(e)} d(w)$.  By the handshake lemma, the latter sum equals $2|B_x(e)|-1$ (the $-1$ is due to the endvertex of $e$ in $B_x(e)$). Similarly, the contribution of $e$ to $D(x)$ is $|B_x(e)|$, from which \eqref{Dpi} easily follows.
\end{proof}

The equivalence of \ref{Bi} and \ref{Bii} in Theorem~\ref{thBev} is now immediate. The fact that \ref{Biii} is also equivalent to these two follows from the following lemma, whose proof is  
straightforward using \eqref{tet} in combination with Lemma~\ref{treelemma}:

\begin{lemma} \label{hit}
Let $x$ and $y$ be two vertices of a tree $T$ with $m$ edges. The hitting time $H_{xy}$ can be expressed as
$$H_{xy} = md(x,y) + D(y) - D(x).$$\qed
\end{lemma}

\comment{
	\begin{align*}
H_{xy} &= \frac12 \sum_{w \in V(T)} d(w)(d(x,y) + d(w,y) - d(w,x)) \\
&= \frac{d(x,y)}{2} \sum_{w \in V(T)} d(w) + \frac12 \left( D_{\pi}(y) - D_{\pi}(x) \right) = m d(x,y) + D(y) - D(x),
\end{align*}
	which proves Theorem~\ref{hit}.
}
Let us now prove Theorem~\ref{th1} in its form~\eqref{CCW}. This can either be achieved by summing Lemma~\ref{hit} over all $y$, which yields
\begin{align*}
CC(x) &= \sum_{y \in V(T)} H_{xy} = m \sum_{y \in V(T)} d(x,y) + \sum_{y \in V(T)} D(y) - n D(x) \\
&= mD(x) + 2W(T)  - nD(x) = 2W(T) - D(x),
\end{align*}
or by specialisation in Theorem~\ref{thMain}:
\begin{align*}
CC(x) &= mR(x) - \frac{n}{2} R_{\pi}(x) + K_{\pi}(T) = mD(x) - \frac{n}{2} D_{\pi}(x) + \frac12 \operatorname{Sch}(T) \\
&= mD(x) - \frac{n}{2} (2D(x)-m) + 2W(T) - \frac{n(n-1)}{2} \\
&= (m-n)D(x) + 2W(T) +  \frac{n(n-1)}{2} - \frac{n(n-1)}{2} = 2W(T) - D(x).
\end{align*}
Analogously, we get
$$RC(x) = (2n-1)D(x) - 2W(T),$$
and Theorem~\ref{RCW} follows immediately. Finally, by Theorem~\ref{thMain} and
Lemma~4.1, we have
\begin{align*}
RC_{\pi}(x) &= 2mR_{\pi}(x) - K_{\pi^2}(T) = 2mD_{\pi}(x) - \operatorname{Gut}(T) \\
&= 2m (2D(x) - m) - (4W(T) - (n-1)(2n-1)) \\
&= 4mD(x) - 2m^2 - 4W(T) + m(2m+1) = 4mD(x) + m - 4W(T),
\end{align*}
which completes the proof of Theorem~\ref{thBev} by showing that \ref{Biv},\ref{Bv} and \ref{Bvi} are indeed equivalent to \ref{Bi}.

\begin{rem}
From Theorem~\ref{RCW} we also see that
$$RC(x) - RC(y) = (2n-1)(CC(y)-CC(x))$$
for any two vertices $x$ and $y$ in a tree, i.e., differences in the reverse cover cost are $2n-1$ times greater than differences in the cover cost.
\end{rem}

\subsection{Random trees} \label{ranT}
Recall that the cover cost was introduced in \cite{cc} as a tractable variant of the cover time. It was shown by Aldous \cite{Aldous} that the cover time of the random walk starting at the root of a uniformly random rooted labelled tree on $n$ vertices is on average of order $n^{3/2}$. Using Theorems~\ref{th1} and~\ref{RCW}, it is easy to obtain analogous and even more precise results for the cover cost and reverse cover cost, building on results of Janson~\cite{Janson} who proved that, for a very general class of trees (simply generated trees or equivalently Galton-Watson trees), the Wiener index and the centrality are of average order $n^{5/2}$ and $n^{3/2}$ respectively. More precisely, if $T_n$ is a random rooted tree from a simply generated class whose root is $r$, and the random variables $W_n$ and $D_n$ are defined by $W_n = W(T_n)$ and $D_n = D(r)$ respectively, then for a certain constant $\alpha$ depending on the specific family of trees (amongst others, this covers the family of labelled trees, the family of binary trees, or the family of plane trees),
$$\big(\alpha n^{-3/2} D_n, \alpha n^{-5/2} W_n\big) \overset{d}{\to} (\xi,\zeta),$$
where the random variables $\xi$ and $\zeta$ can be defined in terms of a normalised Brownian motion $e(t)$, $0 \leq t \leq 1$:
$$\xi = 2\int_0^1 e(t)\,dt \qquad \text{and} \qquad \zeta = \xi - 4 \iint_{0 < s < t < 1} \min_{s \leq u \leq t} e(u) \,ds\,dt.$$
Thus, by \eqref{RCD}, 
if $C_n = CC(r)$ and $R_n = RC(r)$, then
$$\big(\alpha n^{-5/2} C_n, \alpha n^{-5/2} R_n\big) \overset{d}{\to} (2\zeta,2\xi-2\zeta).$$

Moreover, the expectations of $W_n$ and $D_n$ are of order $n^{5/2}$ and $n^{3/2}$ respectively (\cite{EMMS}; see also Janson \cite[Theorem 3.4]{Janson}), from which we deduce, using~\eqref{CCW}, that the expectation of $C_n$ is of order $n^{5/2}$.

More precisely, using \cite[Theorem 3.4]{Janson} we obtain that $\Ex C_n$ is asymptotic to $\sqrt{\pi/2} n^{3/2}$, where the expectation is with respect to the uniformly random rooted labelled tree $T_n$ on $n$ vertices. It is interesting to compare this with a conjecture of \cite[Conjecture~14]{Aldous}, according to which the expected cover and return time $C^+$ of $T_n$ from its root is asymptotic to $6\sqrt{2\pi} n^{3/2}$. Note that the expected cover and return time of any rooted graph is greater than $CC(r) +RC(r)$ \cite{cc}. Moreover, for $T_n$ we have $\Ex CC(r) = \Ex RC(r)$ by linearity of expectation. Putting these facts together, we obtain a lower bound for $\Ex C^+(T_n)$ that is weaker than Aldous' conjecture by a factor of 6:
\begin{corollary}
Let $T_n$ be the uniformly random rooted labelled tree on $n$ vertices. Then $$\Ex C^+(T_n) \gtrsim \sqrt{2\pi} n^{3/2}.$$ 
\end{corollary}
(Here, we write $f(n) \gtrsim g(n)$ if $\liminf f(n)/g(n)\geq 1$.)

\medskip
The above discussion motivates the following question:
\begin{problem}
Ler $r$ be a uniformly chosen random vertex of a random  graph $G$, and let $CT(r)$ denote the cover time from $r$ in $G$. Is it true that the expectations of $\frac{CC(r)}{|V(G)|}$ and $CT(r)$ are of the same asymptotic order?
\end{problem}
Here, we choose \g according to the Erd\H os-Renyi model \cite{ErdosRenyi}, but other random graph distributions can be considered. Note that the cover time and cover cost are by definition not random parameters once $G$ and $r$ are fixed, but expectations; the  randomness  in the problem is introduced by the choice of \g alone, not the behaviour of the \rw.

\subsection{The extremal trees} \label{sec:ext}

In this section we determine the extremal values of hitting time, cover cost and reverse cover cost for trees of given order, making use of the formulas in Theorem~\ref{th1} and Lemma~\ref{hit}.

In view of Lemma~\ref{hit}, hitting times in a tree are always integers, and they trivially satisfy $H_{xy} \geq 1$, with equality if and only if $x$ is a leaf and $y$ its neighbour. The maximum, on the other hand, is (unsurprisingly) obtained for the two ends of a path -- see Corollary~\ref{minmax} below. This is a consequence of the following simple inequality:

\begin{theo}\label{hitbound}
For any two vertices $x$ and $y$ in a tree $T$, we have
$$d(x,y)^2 \leq H_{xy} \leq d(x,y)(2m-d(x,y)).$$
The lower bound holds with equality if and only if, for all vertices $w \in V(T)$, either $w$ lies on the path from $x$ to $y$ or $y$ lies on the path from $w$ to $x$. The upper bound holds with equality if and only if, for all vertices $w \in V(T)$, either $w$ lies on the path from $x$ to $y$ or $x$ lies on the path from $w$ to $y$.
\end{theo}

\begin{proof}
We use formula~\eqref{tet} for the hitting time. By the triangle inequality, we have 
\begin{equation}\label{tri-ineq}
r(x,y) + r(w,y)-r(w,x) = d(x,y)+d(w,y)-d(w,x) \geq 0.
\end{equation}
Moreover, for vertices $w$ that lie on the path $\mathcal{P}$ from $x$ to $y$, we have $d(w,x) = d(x,y) - d(w,y)$ and thus
$$r(x,y) + r(w,y)-r(w,x) = d(x,y)+d(w,y)-d(w,x) = 2d(w,y).$$
It follows that
$$H_{xy} \geq \sum_{w \in \mathcal{P}} d(w)d(w,y).$$
Moreover, $d(w) \geq 2$ for all $w \in \mathcal{P} \setminus \{x,y\}$, and $d(x) \geq 1$, so
$$H_{xy} \geq d(x,y) + 2 \sum_{j=1}^{d(x,y)-1} j = d(x,y)^2,$$
and equality holds if and only if, except for the vertices on $\mathcal{P}$, \eqref{tri-ineq} holds with equality, i.e., for all $w \notin \mathcal{P}$, $y$ lies on the path from $w$ to $x$. This completes the proof of the lower bound, the upper bound immediately follows from~\eqref{com}.
\end{proof}

\begin{cor}\label{minmax}
For any two vertices $x$ and $y$ in a tree $T$ with $m$ edges, we have
$$1 \leq H_{xy} \leq m^2.$$
The lower bound holds with equality if and only if $x$ is a leaf and $y$ its neighbour. The upper bound holds with equality if and only if $T$ is a path and $x$ and $y$ its endpoints.
\end{cor}

Next we turn our attention to the cover cost. In the following two theorems, we determine its minimum and maximum respectively:

\begin{theo}\label{thm:minCC}
The minimum value of $CC(r)$ among all trees of order $n \geq 2$, rooted at a vertex $r$, is $2n^2-6n+5$, and it is only attained by a star, rooted at one of its leaves.
\end{theo}

\begin{proof}
In the following, we use the notation $D_T(r)$ instead of $D(r)$ to emphasize the dependence on the tree $T$. Given the tree $T$, it follows from \Tr{th1} that the minimum of $CC(r)$ is achieved when $D_T(r)$ attains its maximum. Since $D_T(x)$ (as a function of $x$) is convex along paths, this maximum can only be attained when $r$ is a leaf, so we can assume that the root is a leaf in our case. Let $T' = T \setminus r$ be the rest of $T$, and let $r'$ be the unique neighbour of $r$. Then we have
\begin{align*}
CC(r) & = 2W(T) - D_T(r) = 2(W(T') + D_T(r)) - D_T(r) \\
& = 2W(T') + D_T(r) = 2W(T') + |T'| + D_{T'}(r').
\end{align*}
It is well known \cite{dobrynin2001wiener,entringer1976distance} that the Wiener index is minimized by the star $S_n$, so $W(T') \geq W(S_{n-1}) = (n-2)^2$. Moreover, $D_{T'}(r') \geq |T'|-1$ is obvious as well, with equality if and only if $T'$ is a star and $r'$ its centre. It follows that
$$CC(r) \geq 2W(S_{n-1}) + (n-1) + (n-2) = 2(n-2)^2 + 2n-3 = 2n^2-6n+5$$
\fe\ tree $T$ of order $n \geq 2$, with equality if and only if $T$ is the star $S_n$ and $r$ one of its leaves.
\end{proof}

\begin{theo}
The maximum value of $CC(r)$ among all trees of order $n \geq 2$, rooted at a vertex $r$, is $(n^3-n)/3 - \lfloor n^2/4 \rfloor$, and it is only attained by a path, rooted at a midpoint.
\end{theo}

\begin{proof}
Let $r_1,r_2,\ldots,r_k$ be the neighbours of $r$ and let $T_1,T_2,\ldots,T_k$ be the associated branches. Then we have
$$W(T) = \sum_{i=1}^k W(T_i) + \sum_{i=1}^k \sum_{\substack{j=1 \\ j \neq i}}^k (D_{T_i}(r_i) +|T_i|)|T_j| + D_T(r),$$
where the first term accounts for distances between vertices in the same branch, the second term for distances between vertices in different branches, and the last one for distances between the root and other vertices. Moreover,
$$D_T(r) = \sum_{i=1}^k D_{T_i}(r_i) + |T| - 1.$$
Therefore,
$$CC(r) = 2W(T) - D_T(r) = 2\sum_{i=1}^k W(T_i) + 2\sum_{i=1}^k \sum_{\substack{j=1 \\ j \neq i}}^k (D_{T_i}(r_i) +|T_i|)|T_j| + \sum_{i=1}^k D_{T_i}(r_i) + |T|-1.$$
It is known that the Wiener index is maximised by a path \cite{dobrynin2001wiener,entringer1976distance}, and it is also easy to see that $D(r)$ is maximal for a path of which $r$ is an end. Therefore, $CC(r)$ increases if we replace each of the branches $T_i$ by a path with the same number of vertices. This means that we can assume that our tree maximising $CC(r)$ is a subdivided star and $r$ its centre. 

Now assume that $k > 2$ and, \obda, that $|T_1|\leq |T_2|\leq |T_3|$. We claim that if we detach $T_2$ from $r$ and attach it to the last vertex of $T_1$, then $CC(r)$ will increase. To see this, we are going to use the formula
$$CC(r) =\sum_{e \in E(T)} (2|A_r(e)| - 1)|B_r(e)|,$$
which can be deduced from \Tr{th1} and a double-counting argument similar to the one we used in the proof of \Lr{treelemma}. Note that for any edge $e$ not on  $T_1$, the sizes of $A_r(e),B_r(e)$ are not affected by this modification. For an edge $e$ that does lie on $T_1$, its contribution to the sum above changes from $(2A-1)B$ to $(2(A-t)-1)(B+t)$, where $A:= |A_r(e)|, B:= |B_r(e)|$ (as defined for $T$ before the modification) and $t:= |T_2|$. The difference between the two expressions is $$(2(A-t)-1)(B+t) - (2A-1)B = 2At-2tB-2t^2-t= 2t(A-B)-2t(t+\tfrac12),$$
and this is strictly positive \iff\ $A-B>t+\frac12$. Clearly, we have $B \leq |T_1|$ and $A> |T_2|+|T_3|$, and so $A-B> |T_2|+|T_3|- |T_1|\geq|T_2|=t$, where we used our assumption about the sizes of the $T_i$. 
Since all values are integral, we thus obtain $A-B>t+\frac12$ as desired, proving that $CC(r)$ increases when $T_2$ is moved to the end of $T_1$. 

By iterating the argument, we can assume that $T$ is a path.
The minimum of $D(r)$ is clearly attained at a midpoint of the path, and the precise value of $CC(r)$ is easily determined in this case, completing the proof.
\end{proof}

For the reverse cover cost, it is easier to determine the extremal values. We start with the lower bound, which follows immediately from the lower bound in Theorem~\ref{hitbound}.

\begin{theo}
The minimum value of $RC(r)$ among all trees of order $n \geq 2$, rooted at a vertex $r$, is $n-1$, and it is only attained by a star, rooted at its centre. \qed
\end{theo}

As one would expect, the maximum is attained by a path:

\begin{theo}
The maximum value of $RC(r)$ among all trees of order $n \geq 2$, rooted at a vertex $r$, is $n(n-1)(4n-5)/6$, and it is only attained by a path, rooted at one of its ends.
\end{theo}

\begin{proof}
We proceed by induction on $n$. For $n=2$, the statement is trivial. Now let $T$ be a tree of order $n > 2$ and $r$ a vertex for which $RC(r)$ attains its maximum. As in the proof of Theorem~\ref{thm:minCC}, $r$ has to be a leaf. Let $v$ be its neighbour. Then
$$RC(r) = RC_{T \setminus \{r\}}(v) + (n-1)H_{vr},$$
where $RC_{T \setminus \{r\}}(v)$ is the reverse cover cost of $v$ in the reduced tree $T \setminus \{r\}$, since a random walk starting at a vertex other than $r$ has to reach $v$ first before it can reach $r$. By Theorem~\ref{hitbound}, we have $H_{vr} = 2m-1 = 2n-3$, and by the induction hypothesis, $RC_{T \setminus \{r\}}(v) \leq (n-1)(n-2)(4n-9)/6$, with equality if and only if $T \setminus \{r\}$ is a path and $v$ one of its endpoints. Putting the two observations together, we reach the desired result.
\end{proof}

\section{\Tr{thBev} for non-trees} \label{secEx}

Tetali's formula \eqref{tet} shows that condition \ref{Bii} of \Tr{thBev} is still equivalent to \ref{Biii} for general graphs if $D_{\pi}$ is replaced by $R_{\pi}$. Theorem~\ref{thMain} proves that both are equivalent to \ref{Biv} for general graphs as well.
We now construct some examples showing that except for these, all other equivalences fail for non-trees (with $D,D_\pi$ replaced by $R,R_\pi$).

Having seen \Cr{correg}, it is easy to construct examples of non-trees in which inequality \ref{Bvi} of \Tr{thBev} is not equivalent to any of the others. In the regular graph of \fig{pathlike} for example, the functions $R$ and $R_\pi=3R$ are non-constant, and by \Tr{thMain} so are $RC$ and $RC_\pi$.

\showFig{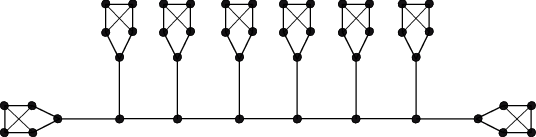}{A regular graph with non-constant $R_\pi,R, RC$ and $RC_\pi$.}

Our next example shows  that \ref{Bi} is not equivalent to any of the other inequalities either: in the graph of \fig{karos}, we have $R(x)=R(y)$ but $R_\pi(x)\neq R_\pi(y)$ as the reader will easily check. Combined with \Tr{thMain}, this implies that $RC(x)\neq RC(y)$ and $CC(x)\neq CC(y)$.

\showFig{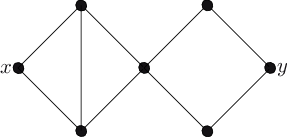}{$R(x)=R(y)$ but $R_\pi(x)\neq R_\pi(y)$.}

Finally, in order to show that \ref{Bv} is not equivalent to \ref{Biv}, it suffices by \Tr{thMain} to have an example in which $R(x)-R(y)> R_\pi(y) -R_\pi(x)>0$. The graph of \fig{brush} is such an example: $R(x)-R(y)$ equals the number of vertices in the star minus the number of vertices in the clique, and so $R(x)-R(y)= l-k$. Similarly, $R_\pi(y)- R_\pi(x)$ equals the sum of degrees in the clique minus the sum of degrees in the star, which is $k(k-1)/2-l$. Letting e.g.\ $l=k(k+1)/2$, we therefore obtain $R_\pi(y) -R_\pi(x) = k$, yielding  the desired $R(x)-R(y)> R_\pi(y) -R_\pi(x)>0$.

\showFig{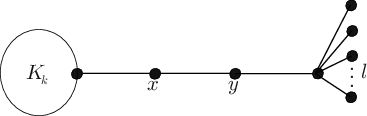}{A graph showing that \ref{Bv} is not equivalent to \ref{Biv}. The circle on the left stands for a $k$-vertex clique joined to $x$ by an edge, while an $l$-vertex star is attached to $y$.}

\bibliographystyle{amsalpha}
\bibliography{collective}

\end{document}

%% file: defs.tex
\usepackage[usenames,dvips]{color} 
\usepackage{amsthm,amssymb,amsmath,enumerate,graphicx,epsf,mathbbol,stmaryrd}
\usepackage[dvips, bookmarks, colorlinks=false, breaklinks=true]{hyperref}

\newcommand{\comment}[1]{}
\newcommand{\COMMENT}[1]{}

\definecolor{darkgray}{rgb}{0.3,0.3,0.3}
\newcommand{\defi}[1]{{\color{darkgray}\emph{#1}}}

\comment{
	\begin{lemma}\label{}	
\end{lemma}
\begin{proof}

\end{proof}

\begin{theorem}\label{}
\end{theorem} 
\begin{proof} 	

\end{proof}

}

\newtheorem{proposition}{Proposition}[section]

\newtheorem{theorem}[proposition]{Theorem}
\newtheorem{corollary}[proposition]{Corollary}
\newtheorem{lemma}[proposition]{Lemma}

\newtheorem{conjecture}{{Conjecture}}[section]

\newtheorem{problem}[conjecture]{{Problem}}

\newtheorem{examp}[proposition]{Example}

\newcommand{\FIG}{0}

\ifnum \NOTESON = 1 \newcommand{\note}[1]{ 

	\ 

	{\color{blue} \hspace*{-60pt} NOTE: \color{Turquoise}{\small  \tt \begin{minipage}[c]{1.1\textwidth}  #1 \end{minipage} \ignorespacesafterend }} 
	
	\ 
	
	}
\else \newcommand{\note}[1]{} \fi

\newcommand{\afsubm}[1]{ \ifnum \Debug = 1 {\mymargin{#1}}
\fi} 

\ifnum \Debug = 1 
\else  \fi

\ifnum \FIG = 1 \newcommand{\fig}[1]{Figure ``{#1}''}
\else \newcommand{\fig}[1]{Figure~\ref{#1}} \fi

\ifnum \FIG = 1 
\else  \fi

\ifnum \Debug = 1 \usepackage[notref,notcite]{showkeys}
\fi

\ifnum \COLORON = 0 \renewcommand{\color}[1]{}
\fi

\newcommand{\showFig}[2]{
   \begin{figure}[htbp]
   \centering
   \noindent
   \epsfbox{#1.eps}
   \caption{\small #2}
   \label{#1}
   \end{figure}
}

\newcommand{\R}{\ensuremath{\mathbb R}}

\newcommand{\pth}[2]{\ensuremath{#1}\text{--}\ensuremath{#2}~path}

\newcommand{\flo}[2]{\ensuremath{#1}\text{--}\ensuremath{#2}~flow} 

\newcommand{\g}{\ensuremath{G\ }}
\newcommand{\G}{\ensuremath{G}}

\newcommand{\kcl}{Kirchhoff's cycle law}

\newcommand{\Ex}{\mathbb E}

\newcommand{\Lr}[1]{Lemma~\ref{#1}}
\newcommand{\Tr}[1]{Theorem~\ref{#1}}
\newcommand{\Sr}[1]{Section~\ref{#1}}

\newcommand{\Cr}[1]{Corollary~\ref{#1}}

\renewcommand{\iff}{if and only if}
\newcommand{\fe}{for every}
\newcommand{\Fe}{For every}

\newcommand{\st}{such that}

\newcommand{\obda}{without loss of generality}

\newcommand{\rw}{random walk}

\newcommand{\labequ}[2]{ \begin{equation} \label{#1} #2 \end{equation} }

\newcommand{\labtequc}[2]{ \begin{equation} \label{#1} 	\text{   #2 } \end{equation} }

\newcommand{\mymargin}[1]{
  \marginpar{%
    \begin{minipage}{\marginparwidth}\small%
      \begin{flushleft}%
        {\color{blue}#1}%
      \end{flushleft}%
   \end{minipage}%
  }%
}%

\newcommand{\mySection}[2]{}

